\documentclass[12pt]{amsart}

\usepackage{graphicx,amssymb,epstopdf,subcaption,mathtools,tikz-cd,pb-diagram,thmtools, thm-restate,chngcntr,pinlabel,pifont,enumitem}

\usepackage[mathscr]{euscript}
\newcommand{\nc}{\newcommand}
\nc{\dmo}{\DeclareMathOperator}
\nc{\nt}{\newtheorem}

\nt{theorem}{Theorem}[section]
\nt{corollary}[theorem]{Corollary}
\nt{lemma}[theorem]{Lemma}
\nt{criterion}[theorem]{Criterion}
\nt*{theorem*}{Theorem}
\nt{question}[theorem]{Question}
\nt{problem}[question]{Problem}
\nt{conjecture}[question]{Conjecture}
\nt{proposition}[theorem]{Proposition}
\nt{example}[theorem]{Example}

\counterwithin{figure}{section}

\dmo{\Map}{Map}
\dmo{\SMod}{SMod}

\nc{\Z}{\mathbb Z}
\nc{\R}{\mathbb R}
\nc{\N}{\mathbb N}

\nc{\C}{\mathcal C}
\nc{\G}{\mathcal G}
\nc{\M}{\mathcal M}
\nc{\f}{\mathfrak f}
\nc{\E}{E} 
\nc{\GE}{E^g}
\nc{\g}{\text{g}}
\newcommand{\cp}{\mathcal P}
\nc{\cv}{\mathcal{V}}
\DeclareMathOperator{\genus}{genus}
\nc{\cantor}{\mathscr{C}}
\nc{\sm}{\setminus}

\usepackage{fullpage}

\nc{\margin}[1]{\marginpar{\scriptsize #1}}

\title{Curve graphs of surfaces with finite-invariance index 1}
\author{Justin Lanier}
\author{Marissa Loving}

\address{Justin Lanier \\ Department of Mathematics\\ University of Chicago \\ 5734 S. University Ave. \\
Chicago, IL, 60637 \\  jlanier@math.uchicago.edu}

\address{Marissa Loving \\ School of Mathematics\\ Georgia Institute of Technology \\ 686 Cherry St. \\ Atlanta, GA, 30332 \\ mloving6@gatech.edu}

\thanks{The first author acknowledges support from NSF Grants DGE-1650044 and DMS-2002187. The second author acknowledges support from NSF Grants DGE-1144245 and DMS-1902729, as well as from NSF Grants DMS 1107452, 1107263, 1107367 ``RNMS: GEometric structures And Representation varieties", which allowed the author to work on this project during the 2017-2018 Warwick EPSRC Symposium on Geometry, Topology and Dynamics in Low Dimensions.}

\begin{document}

\maketitle


\begin{abstract}
In this note we make progress toward a conjecture of Durham--Fanoni--Vlamis, showing that every infinite-type surface with finite-invariance index $1$ and no nondisplaceable compact subsurfaces fails to have a good curve graph, that is, a connected graph where vertices represent homotopy classes of essential simple closed curves and where the natural mapping class group action has infinite diameter orbits. Our arguments use tools developed by Mann--Rafi in their study of the coarse geometry of big mapping class groups.
\end{abstract}


\vspace*{0in}

\section{Introduction}

For the purposes of this note, surfaces are connected orientable 2-manifolds and curves are homotopy classes of essential simple closed curves. A surface $S$ is of finite or infinite type according to whether or not $\pi_1(S)$ is finitely generated. The mapping class group $\Map(S)$ is the group of homotopy classes of orientation-preserving homeomorphisms of $S$.

A central object in the study of finite-type surfaces and their mapping class groups is the curve complex $\C(S)$, which was introduced by Harvey \cite{Harvey}. The vertices of $\C(S)$ represent curves on $S$ and sets of vertices that have disjoint representatives bound simplices. The curve complex is flag and it is also referred to as the curve graph. Many variants of $\C(S)$ have also been considered---for example, by replacing curves with subsurfaces or by considering only separating or nonseparating curves. When $S$ is of finite type, $\C(S)$ has infinite diameter when equipped with the path metric induced by letting each edge have length one, and its geometry is Gromov hyperbolic, as shown by Masur--Minsky \cite{MM}. Further, the orbits of the natural action of $\Map(S)$ on $\C(S)$ have infinite diameter, and analyzing these orbits gives information about the Nielsen--Thurston type of a mapping class.

For any infinite-type surface $S$, it is not hard to see that $\C(S)$ has diameter 2. If we have any hope of recovering the interesting mapping class group actions found in the finite-type setting, it is necessary to build alternatives to $\C(S)$. One option is to replace curves with other objects; in this direction, Calegari and Bavard initiated the study of ray graphs \cite{Calegari, Bavard}. In this note, following Durham--Fanoni--Vlamis, we consider alternatives where vertices correspond to curves on $S$, but where perhaps only a proper subset of curves are included, and where edges need not correspond to disjointness \cite{DFV}. We will say that a graph $\Gamma$ is a curve graph for $S$ if the vertices of $\Gamma$ represent some (not necessarily proper) subset of curves on $S$ and if the natural action of $\Map(S)$ on the vertices of $\Gamma$ induces an action on $\Gamma$ by simplicial automorphisms. When a curve graph $\Gamma$ for $S$ is connected and has an infinite diameter orbit under the natural action of $\Map(S)$, we will call $\Gamma$ a good curve graph for $S$.

Durham--Fanoni--Vlamis defined an invariant of infinite-type surfaces that in most cases determines whether or not $S$ has a good curve graph \cite{DFV}. The finite-invariance index $\f(S)$ of a surface is the size of the largest finite $\Map(S)$-invariant collection of disjoint closed proper subsets of ends of $S$; see Section 2 for more details. Using this invariant, Durham--Fanoni--Vlamis characterized in most cases whether $S$ has a good curve graph.

\begin{theorem}\cite[Main Theorem]{DFV}
\label{thm:DFV}
If $\f(S) \geq 4$, then $\Map(S)$ admits an unbounded action on a graph consisting of curves. If $\f(S) = 0$, then $\Map(S)$ admits no such unbounded action.
\end{theorem}

In earlier work, Aramayona--Valdez proved whether or not $\Map(S)$-invariant subgraphs of $\C(S)$ are good for certain classes of infinite-type surfaces \cite{AV}. Their results break into two cases, depending on whether $S$ has finite or infinite genus. We summarize their relevant results as follows.

\begin{theorem}\cite[Theorems~1.4 and 1.7]{AV}
\label{thm:AV}
Let $S$ be an infinite-type surface.
\begin{enumerate}
    \item If $S$ has finite genus and no isolated punctures, then a $\Map(S)$-invariant subgraph $\G(S)$ has infinite diameter if and only if it contains no separating curves that cut off a disk containing some, but not all, of the punctures.
    \item  If $S$ is a blooming Cantor tree and $\G(S)$ is a $\Map(S)$-invariant subgraph of $\C(S)$, then $diam(\G(S)) = 2$.
\end{enumerate}
\end{theorem}

The statement given in (2) is a correction of the statement given by Aramayona--Valdez \cite[Theorem~1.7]{AV}; their original statement contradicts Theorem~\ref{thm:DFV}. The proof that they give assumes not only that $S$ has infinite genus, but that every end is accumulated by genus \cite{JaviEmail}; this further hypothesis implies that $S$ is a blooming Cantor tree. The surfaces treated in Theorem~\ref{thm:AV} are also treated in Theorem~\ref{thm:DFV}, since they all have finite-invariance index either 0 or $\infty$.

In their paper, Durham--Fanoni--Vlamis make some observations about the cases of finite-invariance index $1$, $2$, and $3$. For $\f(S) = 2$ and $\f(S) = 3$, they give examples showing that finite-invariance index is too coarse an invariant to determine whether $S$ has a good curve graph. On the other hand, they conjecture that $\f(S)=1$ implies that $S$ has no good curve graph \cite[Conjecture 9.1]{DFV}.

The main result of this note confirms their conjecture under the assumption $S$ has no nondisplaceable compact subsurfaces.

\begin{theorem}
\label{thm:f1}
Let $S$ be a surface of infinite type with $\f(S)=1$ such that $S$ has no nondisplaceable compact subsurfaces. Then $S$ does not have a good curve graph.
\end{theorem}

We give two proofs of this theorem. After reviewing some background in Section 2, we begin Section 3 by proving a preliminary lemma that squares up definitions introduced by Durham--Fanoni--Vlamis and Mann--Rafi. Theorem~\ref{thm:f1} then follows directly from a strictly stronger result of Mann--Rafi \cite[Proposition~3.1]{MR}. Our second proof of Theorem~\ref{thm:f1} is given in Section 4. This proof uses the same preliminary lemma but then follows the style of the arguments of Durham--Fanoni--Vlamis.
 
It is an open question whether a surface $S$ with $\f(S)=1$ can have a nondisplaceable compact subsurface. In a previous version of this paper, we claimed a proof that they cannot; Justin Malestein and Jing Tao pointed out an error in our proof. We record this question here, phrased in light of our Lemma~\ref{lem:point-Cantor}:

\begin{question}
Can an infinite-type surface $S$ with either zero or infinite genus have a nondisplaceable compact subsurface when $\M(E)$ consists of a unique equivalence class that is either a point or a Cantor set?
\end{question}

\noindent By Proposition~4.8 of Mann--Rafi, an equivalent formulation asks whether the end space of an infinite-type surface $S$ with either zero or infinite genus can fail to be self-similar if $\M(E)$ consists of a unique equivalence class that is either a point or a Cantor set.

\subsection*{Acknowledgements} The authors thank Javier Aramayona, Federica Fanoni, Justin~Malestein, Katie Mann, Hugo Parlier, Jing Tao, Ferr\'{a}n Valdez, and Nick Vlamis for a number of helpful conversations and correspondences. They thank Justin Malestein and Jing Tao for identifying an error in a proof in an earlier version of this paper.

\section{Background}

The goal of this section is to overview the aspects of infinite-type surfaces relevant to our work. We review the classification of infinite-type surfaces as well as notions related to the finite-invariance index, the Mann--Rafi partial order, self-similarity, and coarse boundedness of a group. For more comprehensive treatments of these topics, we refer the reader to \cite{overview}, \cite{Richards}, \cite{DFV}, and \cite{MR}.

\subsection*{Classifying infinite-type surfaces}The classification of infinite-type surfaces was first given by Ker\'{e}kj\'{a}rt\'{o} and was clarified and extended by Richards \cite{Ker,Richards}.

The classification implies that every connected orientable 2-manifold may be constructed as follows. Begin with a sphere. Puncture it, without loss of generality, along some closed subset of a Cantor set. This set of punctures is denoted $\E(S)$, and punctures are called ends. Finally, add handles to the surface so that the only accumulation points of sequences of handles are in $\E(S)$. This marks some closed subset of $\E(S)$ as being accumulated by genus. This space of genus ends is denoted $\GE(S)$ and is recorded as a subspace of $\E(S)$. Note that $\GE(S)$ is also a closed subset of a Cantor set. Observe that closed surfaces and punctured surfaces of finite type are special cases of this construction.

Let $\g(S) \in \N \cup \{\infty\}$ equal the genus of $S$. The classification states that the triple $(\g(S),\E(S), \GE(S))$ uniquely determines $S$ up to homeomorphism. For more details on the classification and on the definition of the spaces of ends, see \cite{DFV}.

\subsection*{The finite-invariance index} Following Durham--Fanoni--Vlamis, we make the following definitions.
We say that a collection $\cp$ of disjoint subsets of the space of ends is $\Map(S)$-invariant if for every $P\in\cp$ and for every $\varphi\in\Map(S)$ there exists $Q\in\cp$ such that $\varphi(P)=Q$. The finite-invariance index of $S$, denoted $\f(S)$, is defined as follows:

\begin{itemize}
\item $\f(S)\geq n$ if there is a $\Map(S)$-invariant collection $\mathcal{P}$ of disjoint closed proper subsets of $E(S)$ satisfying $|\mathcal{P}| = n$;
\item $\f(S)=\infty$ if $\g(S)$ is finite and positive; 
\item $\f(S)=0$ otherwise. 
\end{itemize}

\noindent We say that $\f(S)=n$ if $\f(S)\geq n$ but $\f(S)\ngeq n+1$.

For any $\Map(S)$-invariant collection $\mathcal{P}$ of disjoint closed proper subsets of $E(S)$, call the elements of $\mathcal{P}$ finite-invariance sets. When $\mathcal{P}$ contains only one set, we also call $\mathcal{P}$ a finite-invariance set.

\subsection*{A partial order and self-similarity} Following Mann--Rafi, we make the following definitions. For $x,y \in E(S)$, we say $x \preccurlyeq y$ if every clopen neighborhood of $y$ contains a homeomorphic copy of some clopen neighborhood of $x$. We say $x$ and $y$ are equivalent if $x \preccurlyeq y$ and $y \preccurlyeq x$. We will make use of the following result concerning this partial order on equivalence classes of ends.

\begin{proposition}[\cite{MR}, Proposition 4.7] \label{prop:maximal}
The partial order $\preccurlyeq$ has maximal elements. Furthermore, the equivalence class of every maximal element is either finite or a Cantor set.
\end{proposition}

\noindent We let $\mathcal M(E)$ denote the set of equivalence classes of maximal ends of $E(S)$.

Continuing to follow Mann--Rafi, a pair $(E,E^g)$ is self-similar if for any decomposition $E=E_1 \sqcup E_2 \sqcup \dots \sqcup E_n$ of $E$ into pairwise disjoint clopen sets, there exists a clopen set $D$ contained in some $E_i$ such that the pair $(D, D \cap E^g)$ is homeomorphic to $(E, E^g)$.

A Polish group $G$ is (globally) coarsely bounded, or CB, if every compatible left-invariant metric on $G$ gives $G$ finite diameter. Mann--Rafi proved the following sufficient condition for $\Map(S)$ to be CB.

\begin{proposition}[\cite{MR}, Proposition 3.1] \label{prop:self-similar}
Let $S$ be a surface of infinite or zero genus. If the space of ends of $S$ is self-similar, then $\Map(S)$ is CB. 
\end{proposition}

\noindent A helpful characterization of CB groups, which we will utilize in Section~\ref{section:main-thm-first}, relies on the notion of a length function. A length function on a topological group $G$ is a continuous function $\ell: G \to [0, \infty)$ such that $\ell(g) = \ell(g^{-1})$, $\ell(id) = 0$, and $\ell (gh) \leq \ell(g) + \ell(h)$ for all $g, h \in G.$ It follows from work of Rosendal \cite[Theorem 10]{Rosendal} that a topological group $G$ which admits an unbounded length function $\ell$ is not CB.

\subsection*{Good curve graphs vs. coarse boundedness}

If a surface $S$ has a good curve graph, then $\Map(S)$ is not CB. However, the converse is not true. There are surfaces whose mapping class groups are not CB, and so admit an action with an unbounded orbit on some metric space, yet do not admit such an action in a natural way on any metric space arising as a connected curve graph for the surface.

Examples of such surfaces include the plane minus a Cantor set and the tripod surface, which has exactly three ends, all accumulated by genus. These surfaces have finite-invariance index 2 and 3, respectively. Since these surfaces have nondisplaceable subsurfaces, their mapping class groups are not CB by the work of Mann--Rafi \cite[Theorem~1.9]{MR}. However, Durham--Fanoni-Vlamis showed that neither surface has a good curve graph \cite[Proposition~9.2]{DFV}.

In cases where such a ``gap" exists, we can hope to find a substitute for graphs of curves by constructing graphs of other objects naturally associated to a surface. For instance, in the case of the plane minus a Cantor set, Bavard showed that the ray graph is connected and  $\delta$-hyperbolic and has infinite diameter orbits under the natural action of the mapping class group \cite{Bavard}. In the case of the tripod, it follows from a recent result of Fanoni--Ghaswala--McLeay that its mapping class group acts with infinite diameter orbits on its omnipresent arc graph; this is a connected, $\delta$-hyperbolic graph where vertices correspond to certain (possibly bi-infinite) arcs between distinct ends of the surface and where edges correspond to disjointness \cite[Theorem~C]{GFM2021}. 

\section{Main theorem: first approach}
\label{section:main-thm-first}

In this section we begin by relating the finite-invariance index and the partial order on equivalence classes of ends in the case $\f(S) = 1$. We then show that if $S$ has a good curve graph, then $\Map(S)$ has an unbounded length function. We close by giving our first proof of Theorem~\ref{thm:f1}.

\begin{lemma} \label{lem:point-Cantor}
Let $S$ be a surface of infinite type with $\f(S)=1$. Then $\mathcal{M}(E)$ consists of a unique equivalence class that is either a point or a Cantor set.
\end{lemma}

\begin{proof}
Let $S$ be a surface with $\f(S)=1$ and with end space $E$. It follows that $\mathcal M(E)$ consists of a unique equivalence class, for otherwise two of the equivalence classes could be taken as a collection of finite-invariance sets for $S$, as these would be closed, $\Map(S)$-invariant,  proper, and disjoint, contradicting $\f(S)=1$.

By Proposition \ref{prop:maximal}, an individual maximal equivalence class is either finite or a Cantor set. If $\mathcal M(E)$ is finite, it must consist of a single point. Otherwise each of the points could be taken as a finite-invariance set for $S$, as these would be closed, $\Map(S)$-invariant,  proper, and disjoint, contradicting $\f(S)=1$. Therefore $\mathcal{M}(E)$ is either a point or a Cantor set.
\end{proof}

\begin{proposition}
\label{prop:lengthfunction}
If $S$ is an infinite-type surface with a good curve graph $\Gamma$, then $\Map(S)$ has an unbounded length function.
\end{proposition}

\begin{proof}
Define $\ell: \Map(S) \rightarrow \mathbb{Z}_{\geq 0}$ by $\ell(f)=d(f(c),c)$ for $f \in \Map(S)$, $d$ the path metric in $\Gamma$, and $c$ a fixed base point in $\Gamma$ such that the orbit of $c$ under the natural action of $\Map(S)$ has infinite diameter. We first verify that $\ell$ is a length function. Since the action of $\Map(S)$ on $\Gamma$ is isometric, it follows that $\ell(f)=\ell(f^{-1})$ and $\ell(id)=0$. Since $\Gamma$ is path connected, the value $\ell(f)$ is finite and non-negative for all $f$. The inequality $\ell(gh) \leq \ell(g) + \ell(h)$ follows from the triangle inequality in $\Gamma$. Finally, we have that $\ell$ is continuous, since for any given $f \in \Map(S)$, the preimage of $\ell(f)$ under $\ell$ is open; it consists of a union of open sets in $\Map(S)$ in its compact-open topology, namely the open sets $\{h \in \Map(S) \ | \ h(c)=g(c) \}$ for all $g$ where $\ell(g)=\ell(f)$. Since the orbit of $c$ under the natural action of $\Map(S)$ has infinite diameter, $\ell$ is unbounded.
\end{proof}

\begin{proof}[Proof \#1 of Theorem~\ref{thm:f1}]
By Lemma~\ref{lem:point-Cantor}, $\M(E)$ consists of a unique equivalence class that is either a point or a Cantor set. Since by hypothesis $S$ has no nondisplaceable compact subsurfaces, Proposition~4.8 of Mann--Rafi implies that $E(S)$ is self-similar. By Proposition~\ref{prop:self-similar}, this implies that $\Map(S)$ is CB. It follows from a theorem of Rosendal that a CB group cannot have an unbounded length function; see Section 2 of Mann--Rafi. We conclude by Proposition~\ref{prop:lengthfunction} that $S$ has no good curve graph.
\end{proof}

\section{Main theorem: second approach}

In this section we will first review the approach introduced by Durham--Fanoni--Vlamis to show that a surface has no good curve graph. We then extend their approach under the hypothesis $\f(S)=1$. We conclude with our second proof of Theorem~\ref{thm:f1}.

Durham--Fanoni--Vlamis introduced the following tool for showing that a curve graph is not good.

\begin{proposition}[\cite{DFV}, Proposition~4.1]\label{prop:criterionfd}
Let $S$ be an oriented surface and $\Gamma = \Gamma(S)$ be a connected graph consisting of curves on which the mapping class group acts. Let $\cv \subset \Gamma \times \Gamma$ satisfying:
\begin{enumerate}
\item there exists a vertex $c\in \Gamma$ such that, up to the action of $\Map(S)$, there is a finite number of pairs $(a,b)\in\cv$ with $a,b \in \Map(S)\cdot c$, and
\item for every $a,b\in \Map(S)\cdot c$ with $(a,b)\notin\cv$, there exists $d \in \Map(S)\cdot c$ such that $(a,d)$ and $(b,d)$ belong to $\cv$.
\end{enumerate}  Then every $\Map(S)$-orbit in $\Gamma$ has finite diameter.
\end{proposition}

Durham--Fanoni--Vlamis use this tool to derive three easily-applied criteria for showing that a curve graph is not good. They use these criteria to show that surfaces with finite-invariance index 0 do not have good curve graphs. These three criteria are given in the following proposition.

\begin{proposition}[\cite{DFV}, Proposition~4.2]\label{prop:conditions}
Suppose that $\Gamma=\Gamma(S)$ is a connected graph consisting of curves with an action of $\Map(S)$ and:
\begin{enumerate}
\item
$S$ has infinitely many isolated punctures and $\Gamma$ contains a vertex bounding a finite-type genus-0 surface, or
\item $\genus(S)=\infty$, $S$ has either no punctures or infinitely many isolated punctures (and in the latter case has an end accumulated by both genus and punctures), and $\Gamma$ contains a curve bounding a finite-type surface, or
\item $\genus(S)=\infty$ and $\Gamma$ contains a nonseparating curve,
\end{enumerate}
then $\Map(S)$ acts on $\Gamma(S)$ with finite-diameter orbits.
\end{proposition}

The statement given in (2) adds a parenthetical hypothesis to the statement given by Durham--Fanoni--Vlamis. Without some hypothesis on the positioning of the genus and punctures in the case that $S$ has infinite genus and infinitely many punctures, their proof of their Proposition~4.2(2) does not go through. For instance, take $S$ to be a surface with end space consisting of one end accumulated by genus and a second end accumulated by isolated punctures. Let $c$ be a separating curve that bounds a subsurface of genus 1 with one puncture. Then there are curves $a,b \in \Map(S) \cdot c$ such that $a$ and $b$ fill a separating finite-type subsurface $F$ where one component of $S \setminus F$ has no punctures and the other component has no genus. Because of this there is no curve $d \in \Map(S) \cdot c$ that is disjoint from both $a$ and $b$, and so Proposition~\ref{prop:criterionfd} does not apply directly. The additional hypothesis we introduce avoids this situation.

In all of the cases where Durham--Fanoni--Vlamis apply their Proposition~4.2(2), the further hypothesis in fact holds, and so the further results they prove that rely on the proposition remain true. This hypothesis also holds in all of the cases we consider in the present note.

The upshot of Proposition~4.2 is that certain curves in $S$ are ``bad", in the sense that no good curve graph for $S$ can contain such a curve. For instance, under certain hypotheses, separating curves are bad when they cut off a finite-type surface from an ``infinite pool" of genus and/or isolated punctures. In our proof of  Propositions~\ref{prop:newbadcurvesfinal} we show that, under the hypotheses of $\f(S)=1$ and self-similarity, separating curves are also bad when they cut off an infinite-type subsurface, as they leave behind in one complementary component an ``infinite pool" of subsurfaces homeomorphic to the other complementary component.

We will require a further proposition, which follows from a result of Malestein--Tao.

\begin{proposition}
\label{prop:maxselfsim}
Suppose $E(S)$ is self-similar and that $C$ is a clopen subset of $E$(S) containing a maximal end of $E(S)$. Then $C$ is homeomorphic to $E(S)$. 
\end{proposition}

\begin{proof}
Let $x \in C \subseteq E$ be a maximal end of $E$. By Lemma~2.8 of Malestein--Tao, every clopen neighborhood of a maximal end in a self-similar end space is homeomorphic \cite{JJ}. Since $E$ and $C$ are clopen neighborhoods of $x$, they are homeomorphic.
\end{proof}

A noteworthy consequence of this proposition is that the following definition of self-similarity is equivalent to the Mann--Rafi definition of self-similarity: a pair $(E,E^g)$ is self-similar if for any decomposition $E=E_1 \sqcup E_2 \sqcup \dots \sqcup E_n$ of $E$ into pairwise disjoint clopen sets, for some $E_i$ the pair $(E_i, E_i \cap E^g)$ is homeomorphic to $(E, E^g)$. This is because some $E_i$ must contain a maximal end.

For a surface $S$ with a self-similar end space $E$, we call two distinct disjoint separating curves unnested if the end space of the subsurface they cobound contains an end in $\mathcal M(E)$.  Otherwise, a pair of distinct disjoint separating curves is nested.

\begin{proposition}
\label{prop:newbadcurvesfinal}
Suppose $S$ is of infinite type with $\f(S)=1$ and that $E(S)$ is self-similar. Suppose further that that $\Gamma$ is a connected curve graph for $S$. Then $\Map(S)$ acts on $\Gamma$ with finite-diameter orbits.
\end{proposition}

\begin{proof}
Let $S$ and $\Gamma$ be as in the statement. Let $c$ be a curve in $\Gamma$. Since $\f(S)=1$, we have $\g(S)=0$ or $\infty$. If $c$ is a nonseparating curve, then $\g(S)=\infty$, Proposition \ref{prop:conditions}(3) applies, and $\Map(S)$ acts on $\Gamma$ with finite-diameter orbits.

If $c$ is separating and cuts off a finite-type surface, then $\Map(S)$ acts on $\Gamma$ with finite-diameter orbits since either Proposition \ref{prop:conditions}(1) or (2) applies, as follows. Since we already have that $\g(S)=0$ or $\infty$, we next show that $S$ has either 0 or infinitely many isolated punctures. The surface $S$ cannot have two or more (but finitely-many) isolated punctures, since these would be maximal, contrary to the condition on $\M(E)$ granted by Lemma~\ref{lem:point-Cantor}. If there is a single isolated puncture, then it must be maximal and so equal to $\M(E)$ and in fact all of $E$. But then $S$ is the plane and so not of infinite type. Finally, if $S$ has infinite genus and infinitely many isolated punctures, then every maximal end is accumulated both by genus and by isolated punctures by maximality. We conclude that the proposition applies.

Otherwise, $c$ is separating but does not cut off a finite-type subsurface.  Set \[ \mathcal V = \{(a, b)~|~ a \text{ and } b \text{ are unnested}\}.\]

\noindent There is a unique pair of separating curves $(a,b) \in \cv$ with $a, b \in \Map(S) \cdot c$, up to the action of $\Map(S)$. This follows from  Proposition~\ref{prop:maxselfsim} and the classification of surfaces. Therefore condition (1) of Proposition \ref{prop:criterionfd} is satisfied.

Next, consider $a, b \in \Map(S) \cdot c$ such that $(a, b) \not \in \cv$. This means that either $a$ and $b$ intersect, or they are nested, or they are not distinct. In each case we will find $d \in \Map(S) \cdot c$ such that $(a,d)$ and $(b,d)$ belong to $\cv$. If $a$ and $b$ intersect, then together they fill a finite-type subsurface $F$. By Proposition~\ref{prop:maxselfsim}, one of the components of $S \setminus F$ has end space homeomorphic to $E(S)$ since its end space contains a maximal end. Then there exists a separating curve $d$ in this component that is in $\Map(S) \cdot c$ and unnested with each of $a$ and $b$. Therefore, $(a,d), (b,d) \in \cv$. Similarly, if $a$ and $b$ are nested, there is again a component of $S \setminus \{a \cup b\}$ with end space homeomorphic to $E(S)$ where there exists the desired curve $d$ that is unnested with each of $a$ and $b$. A similar argument applies if $a$ and $b$ are not distinct. Condition (2) of Proposition \ref{prop:criterionfd} is therefore satisfied, and $\Map(S)$ acts on $\Gamma$ with finite-diameter orbits.

Since this exhausts the possibilities for $c$, we conclude that $\Map(S)$ acts on $\Gamma$ with finite-diameter orbits and that $S$ has no good curve graph.\end{proof}

\begin{proof}[Proof \#2 of Theorem \ref{thm:f1}]
By Lemma~\ref{lem:point-Cantor}, $\M(E)$ consists of a unique equivalence class that is either a point or a Cantor set. Since by hypothesis $S$ has no nondisplaceable compact subsurfaces, Proposition~4.8 of Mann--Rafi implies that $E(S)$ is self-similar. By applying Proposition \ref{prop:newbadcurvesfinal}, we conclude that $S$ has no good curve graph. \end{proof}

\bibliographystyle{plain}
\bibliography{paper}

\begin{thebibliography}{10}

\bibitem{JaviEmail}
Javier Aramayona.
\newblock {Private communication}, 2018.

\bibitem{AV}
Javier Aramayona and Ferr\'{a}n Valdez.
\newblock On the geometry of graphs associated to infinite-type surfaces.
\newblock {\em Mathematische Zeitschrift}, 2016.

\bibitem{overview}
Javier Aramayona and Nicholas~G. Vlamis.
\newblock Big mapping class groups: an overview.
\newblock \texttt{arXiv:2003.07950}, 2020.

\bibitem{Bavard}
Juliette Bavard.
\newblock Hyperbolicit\'{e} du graphe des rayons et quasi-morphismes sur un
  gros groupe modulaire.
\newblock {\em Geom. Topol.}, 20(1):491--535, 2016.

\bibitem{Calegari}
Danny Calegari.
\newblock “{B}ig mapping class groups and dynamics”, {G}eometry and the
  imagination.
\newblock \\
  \texttt{https://lamington.wordpress.com/2009/06/22/big-mapping-class-groups-and-dynamics/},
  2009.

\bibitem{DFV}
Matthew~Gentry Durham, Federica Fanoni, and Nicholas~G. Vlamis.
\newblock Graphs of curves on infinite-type surfaces with mapping class group
  actions.
\newblock {\em Ann. Inst. Fourier (Grenoble)}, 68(6):2581--2612, 2018.

\bibitem{GFM2021}
Federica Fanoni, Tyrone Ghaswala, and Alan McLeay.
\newblock Homeomorphic subsurfaces and the omnipresent arcs.
\newblock \texttt{arXiv:2003.04750}, 2021.

\bibitem{Harvey}
W.~J. Harvey.
\newblock Boundary structure of the modular group.
\newblock In {\em Riemann surfaces and related topics: {P}roceedings of the
  1978 {S}tony {B}rook {C}onference ({S}tate {U}niv. {N}ew {Y}ork, {S}tony
  {B}rook, {N}.{Y}., 1978)}, volume~97 of {\em Ann. of Math. Stud.}, pages
  245--251. Princeton Univ. Press, Princeton, N.J., 1981.

\bibitem{Ker}
B\'{e}la Ker\'{e}kj\'{a}rt\'{o}.
\newblock {\em Vorlesungen über Topologie. I}.
\newblock Springer, Berlin, 1923.

\bibitem{JJ}
Justin Malestein and Jing Tao.
\newblock Self-similar surfaces: Involutions and perfection.
\newblock \texttt{arXiv:2106.03681}, 2021.

\bibitem{MR}
Kathryn Mann and Kasra Rafi.
\newblock Large scale geometry of big mapping class groups.
\newblock \texttt{arXiv:1912.10914}, 2019.

\bibitem{MM}
Howard~A. Masur and Yair~N. Minsky.
\newblock Geometry of the complex of curves. {I}. {H}yperbolicity.
\newblock {\em Invent. Math.}, 138(1):103--149, 1999.

\bibitem{Richards}
Ian Richards.
\newblock On the classification of noncompact surfaces.
\newblock {\em Trans. Amer. Math. Soc.}, 106:259--269, 1963.

\bibitem{Rosendal}
Christian Rosendal.
\newblock Large scale geometry of metrisable groups.
\newblock \texttt{arXiv:1403.3106}, 2014.

\end{thebibliography}
\end{document}